\documentclass[12pt]{article}
\usepackage{amssymb,amsmath,epsfig,amscd,eucal,psfrag,amsthm}
\usepackage{graphicx}

\pdfoutput=1

\newtheorem{theorem}{Theorem}

\newtheorem{lemma}[theorem]{Lemma}

\newtheorem{conjecture}[theorem]{Conjecture}

\newtheorem{letterthm}{Theorem}

\theoremstyle{definition}
\newtheorem{definition}[theorem]{Definition}

\theoremstyle{remark}
\newtheorem{remark}[theorem]{Remark}

\newcommand{\curlyP}{\mathcal{P}}

\input xy
\xyoption{all}

\title{Alternating quotients of free groups}
\author{Henry Wilton\footnote{Partially supported by NSF grant number DMS-0906276 and by an EPSRC Career Acceleration Fellowship.}}

\begin{document}

\maketitle

\begin{abstract}
We strengthen Marshall Hall's Theorem to show that free groups are locally extended residually alternating.  Let $F$ be any free group of rank at least two, let $H$ be a finitely generated subgroup of infinite index in $F$ and let $\{\gamma_1,\ldots,\gamma_n\}\subseteq F\smallsetminus H$ be a finite subset. Then there is a surjection $f$ from $F$ to a finite alternating group such that $f(\gamma_i)\notin f(H)$ for any $i$.  The techniques of this paper can also provide symmetric quotients. 
\end{abstract}

\begin{definition}\label{d: Residually finite}
A group $\Gamma$ is \emph{residually $\curlyP$} (for some class of groups $\curlyP$) if, for any $\gamma\in\Gamma\smallsetminus 1$, there is a surjection $f:\Gamma\to P\in\curlyP$ with $f(\gamma)\neq 1$.
\end{definition}

Many groups are known to be residually finite, and it is natural to ask whether one can restrict attention to smaller classes of finite groups.  Katz and Magnus proved that free groups are residually alternating, hence residually simple \cite{katz_residual_1968} (see also \cite{pride_residual_1972}, \cite{wiegold_free_1977} and \cite{weigel_residual_1993} \emph{et seq.}); in the topological context, Long and Reid showed that many hyperbolic 3-manifold groups are residually $PSL_2(\mathbb{F}_p)$ \cite{long_simple_1998}. 
   
One of the most obvious generalisations of Definition \ref{d: Residually finite} replaces the trivial subgroup with an arbitrary finitely generated subgroup $H$.  (In many cases, it is too much to expect infinitely generated subgroups to satisfy this property---see Remark \ref{r: ERF}.)  This consideration leads to the following definition.  

\begin{definition}
A group $\Gamma$ is said to be \emph{locally extended residually finite} (\emph{LERF}, also often called \emph{subgroup separable}) if, for any finitely generated subgroup $H\subseteq \Gamma$ and any $\gamma\in\Gamma\smallsetminus H$, there is a surjection $f$ from $\Gamma$ onto a finite group such that $f(\gamma)\notin f(H)$.
\end{definition}

\begin{remark}\label{r: ERF}
Without the requirement that the subgroup $H$ be finitely generated, such a group $\Gamma$ is called \emph{extended residually finite} or \emph{ERF}.  It is well known that the \emph{$(2,3)$-Bausmlag--Solitar group}
\[
BS(2,3)=\langle a,b\mid b^{-1}a^{-2}ba^3\rangle
\]
is not residually finite.  By considering the kernel of the natural map $F_2\cong\langle a,b\rangle\to BS(2,3)$, it follows easily that  non-abelian free groups are not ERF.
\end{remark}

Marshall Hall Jr proved that free groups are locally extended residually finite \cite{hall_jr_subgroups_1949}.  His theorem has been reinterpreted topologically and generalised to much larger classes of groups \cite{scott_subgroups_1978,stallings_topology_1983,wilton_halls_2008,haglund_finite_2008}.  However, to the best of the author's knowledge, no results have been proved that restrict the finite quotients that arise to a more specific class.  The aim of this note is to begin to fill this gap by proving that free groups are what one might call locally extended residually alternating.

\begin{letterthm}\label{t: Alternating quotients}
Let $F$ be a free group of rank greater than one.  Let $H$ be a finitely generated subgroup of infinite index in $F$ and let $\{\gamma_1,\ldots,\gamma_n\}$ be a finite subset of $F\smallsetminus H$.  There is a surjection $f$ from $F$ onto some finite alternating group $A_k$ such that $f(\gamma_i)\notin f(H)$ for all $i$.
\end{letterthm}

A small modification of the argument gives symmetric, rather than alternating, quotients.

\begin{letterthm}\label{t: Symmetric quotients}
Let $F$ be a free group of rank greater than one.  Let $H$ be a finitely generated subgroup of infinite index in $F$ and let $\{\gamma_1,\ldots,\gamma_n\}$ be a finite subset of $F\smallsetminus H$.  There is a surjection $f$ from $F$ onto some finite symmetric group $S_k$ such that $f(\gamma_i)\notin f(H)$ for all $i$.
\end{letterthm}

\begin{remark}
In the case when $\curlyP$ is the class of all finite groups, it is equivalent to state the theorem for a single element $\gamma\notin F\smallsetminus H$ instead of for a finite subset $\{\gamma_1,\ldots,\gamma_n\}\subseteq\Gamma\smallsetminus H$: to deduce the latter from the former, simply take the product of the quotients across all $\gamma_i$.  This works because a product of finite groups is finite.  As a product of alternating groups is not alternating and a product of symmetric groups is not symmetric, we give the stronger statements.
\end{remark}
\begin{remark}
The hypothesis that $H$ is of infinite index in $F$ is necessary.  For instance, suppose that $H$ is a normal subgroup of finite index in $F$.  For any surjection $f:F\to A_n$ with $n\geq 5$, $f(H)$ is a normal subgroup of $A_n$ and is therefore either the whole of $A_n$ or trivial.  In the latter case, it follows that $F/H$ maps onto $A_n$.  But $F$ has many finite quotients that do not map onto $A_n$.
\end{remark}

Stallings interpreted Hall's original proof that free groups are LERF using the topology of graphs, reducing it to the topological fact that any immersion of finite graphs can be completed to a covering map \cite{stallings_topology_1983}.  Wiegold's proof that free groups are residually even alternating \cite{wiegold_free_1977} uses a classical theorem of Jordan, which asserts that if the minimal degree of a primitive permutation group is small enough, then that group must be symmetric or alternating \cite{jordan_theoremes_1871}.  The proofs of Theorems \ref{t: Alternating quotients} and \ref{t: Symmetric quotients} combine Jordan's theorem with the covering theory of graphs.  The key technical result is Lemma \ref{l: Alternating completion}.  In the proof, we show how to complete an immersion of a finite graph into the rose to a covering map in such a way that the resulting permutation action of $F$ on the vertices is primitive and satisfies Jordan's condition.  This proves that free groups are, so to speak, locally extended residually symmetric-or-alternating.  A small modification of this construction then forces the action to be alternating; a slightly different modification forces the action to be symmetric.

In order to apply Jordan's theorem, we must ensure that the action of $F$ on the vertices of the cover we construct is primitive.  In this paper we do so in the simplest possible way, by requiring the number of vertices of the cover to be prime.  Thus, we actually prove that free groups are locally extended residually alternating-of-prime-degree.  Alternatively, one could ensure that the action is primitive by forcing it to be 2-transitive.  This is possible, using a more complicated construction, of which we do not give the details here.  Via this more complicated construction, one can obtain different restrictions on the degrees of the alternating quotients.

After free groups, the fundamental groups of surfaces are the next examples to consider.

\begin{conjecture}
Let $\Sigma$ be a closed, orientable, hyperbolic surface and let $H$ be a finitely generated subgroup of infinite index in $\pi_1\Sigma$.   For any $\gamma\in\pi_1\Sigma\smallsetminus H$ there is a surjection $f$ from $\Gamma$ onto a finite alternating group with $f(\gamma)\notin f(H)$.
\end{conjecture}

In Section \ref{s: Permutations and coverings}, we recall the very well known facts that subgroups of free groups can be viewed in terms of either permutation representations or coverings of graphs, and observe that it is easy to pass from one point of view to the other.  In Section \ref{s: Hall's Theorem} we revisit Stallings's proof of Marshall Hall's theorem.  Finally, in Section \ref{s: Alternating quotients} we bring in Jordan's theorem and prove Theorems \ref{t: Alternating quotients} and \ref{t: Symmetric quotients}.

\section{Permutations and coverings of graphs}\label{s: Permutations and coverings}

Let $F$ be a free group of rank $r>1$, with generators $\{\alpha_1,\ldots,\alpha_r\}$.  It is convenient to take $r$ to be finite, although our results apply just as well to the infinite-rank case.  Let $H$ be a subgroup of finite index.  The action of $F$ by left multiplication on $F/H$ can be thought of as a homomorphism from $F$ to the symmetric group $\mathrm{Sym}(F/H)$.   Of course, $F$ acts transitively, and the subgroup $H$ can be recovered as the stabiliser of the trivial coset.   We summarise this paragraph as follows.

\begin{remark}\label{r: permutations}
Subgroups of $F$ of index $d$ correspond canonically to transitive actions of $F$ on pointed sets of cardinality $d$.
\end{remark}

We now switch to a topological point of view.  Let $X$ be a rose with $r$ petals, that is, a graph with precisely one vertex $x_0$ and $r$ edges. (All our results can be generalised to the situation in which $X$ is an arbitrary finite graph.  However, the case in which $X$ is a rose is sufficient for our purposes.) We fix once and for all an isomorphism $F\cong\pi_1(X,{x_0})$ by orienting each edge of $X$ and labelling it with a generator of $F$.

By standard covering space theory, a subgroup $H$ of $F$ corresponds to a connected, pointed covering space $(Y,y_0)\to (X,x_0)$ with the covering map inducing an isomorphism $\pi_1(Y,y_0)\cong H$.

The orientation and labelling of the edges of $X$ pulls back to an orientation and labelling of the edges of $Y$.  Conversely, any orientation and labelling of the edges of $Y$ determines a combinatorial map $Y\to X$ that sends vertices to vertices and edges to edges.  We will only consider such maps, and we will usually think of them in terms of the corresponding labelled and oriented graph $Y$.

Covering maps to $X$ are easily characterised in terms of the labelling and orientation on $Y$.

\begin{definition}
A labelling and orientation on a graph $Y$ are said to satisfy the \emph{covering condition} if, for each vertex $y$ of $Y$ and each label $\alpha_i$, there is exactly one incoming and one outgoing edge labelled $\alpha_i$ at $y$.
\end{definition}

The proof of the following lemma is a simple exercise.

\begin{lemma}
A combinatorial map of non-empty, connected graphs $Y\to X$ is a covering map if and only if it is bijective on the links of vertices.  Equivalently, $Y\to X$ is a covering map if and only if the corresponding labelling and orientation on the graph $Y$ satisfy the covering condition.
\end{lemma}

When a covering map is restricted to a subgraph, we obtain an immersion.

\begin{definition}
A combinatorial map $Y\to X$ is an \emph{immersion} if it is injective on links of vertices.
\end{definition}

Again, this can be characterised in terms of the labelling and orientation on $Y$.

\begin{definition}
A labelling and orientation on a graph $Y$ are said to satisfy the \emph{immersion condition} if, for each vertex $y$ of $Y$ and each label $\alpha_i$, there is at most one incoming and one outgoing edge labelled $\alpha_i$ at $y$.
\end{definition}

Just as before, the proof of the following lemma is an easy exercise.

\begin{lemma}
A combinatorial map of non-empty, connected graphs $Y\to X$ is an immersion if and only if the corresponding labelling and orientation on the graph $Y$ satisfy the immersion condition.
\end{lemma}

The immersion condition and the covering condition are illustrated in Figure \ref{fig: conditions}.

\begin{figure}
\centering
\includegraphics[width=3.75in]{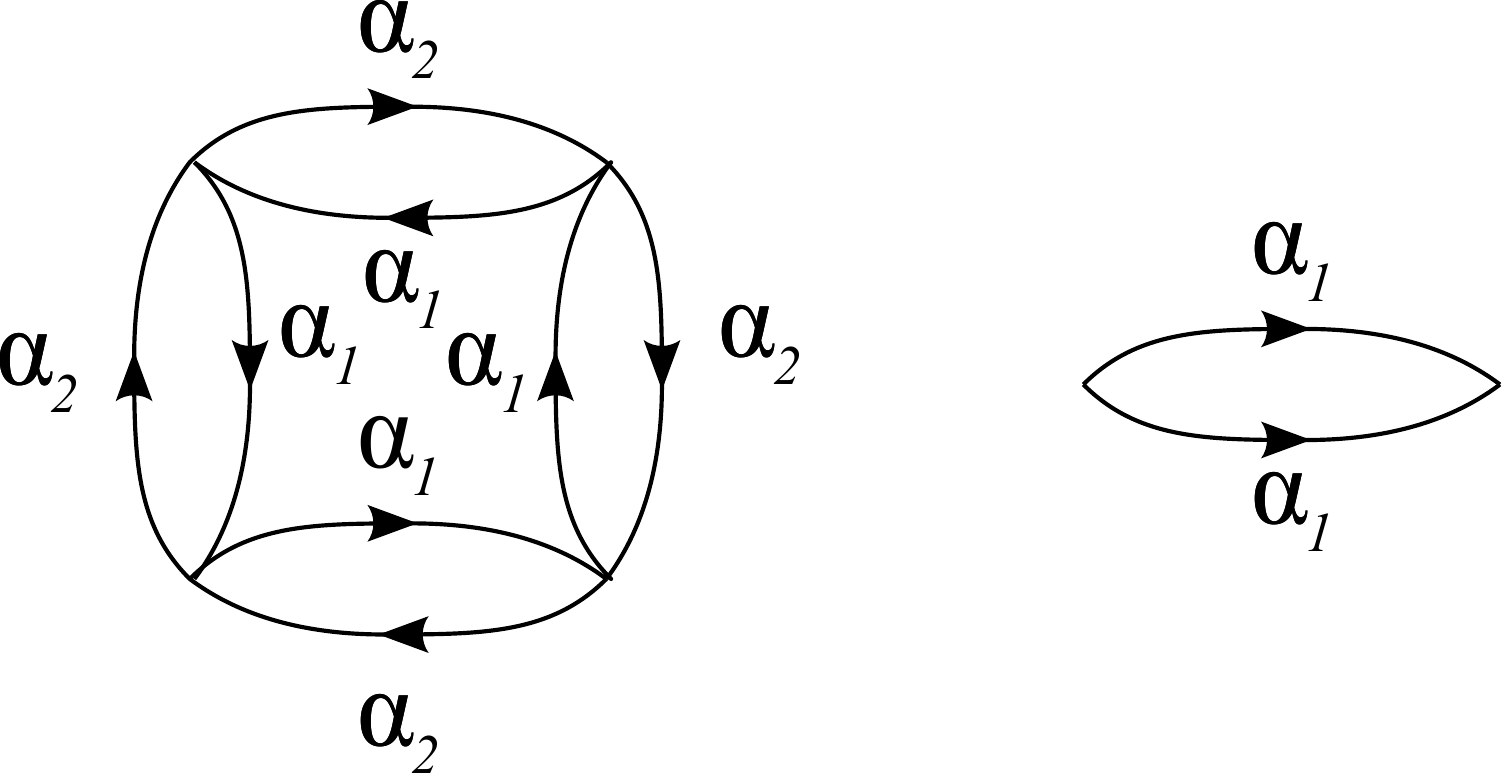}
\caption{The graph on the left satisfies the covering condition. The graph on the right does not satisfy the immersion condition.}
\label{fig: conditions}
\end{figure}

It follows from standard covering space theory that $F$ acts on the vertex set of $Y$ by path lifting.  To be precise, any element $\gamma\in F$ can be thought of as a loop in $X$, represented by a continuous map $\gamma:[0,1]\to X$ with $\gamma(0)=\gamma(1)=x_0$.  Covering space theory asserts that, for each vertex $y\in Y$, there is a unique lift $\tilde{\gamma}:[0,1]\to Y$ with $\tilde{\gamma}(0)=y$.  The action of $\gamma$ on the vertices of $Y$ is defined by
\[
\gamma.y=\tilde{\gamma}(1)~.
\]
Furthermore, the map that sends $\gamma\mapsto \gamma.y_0$ induces an $F$-equivariant bijection between $F/H$ and the vertices of $Y$.  We can summarise this discussion in the following remark.

\begin{remark}\label{r: coverings}
Subgroups of $F$ correspond to connected, based, labelled, oriented graphs that satisfy the covering condition.  The index of the subgroup is equal to the number of vertices of the graph.
\end{remark}

It is apparent from Remarks \ref{r: permutations} and \ref{r: coverings} that connected, based, labelled, oriented graphs satisfying the covering condition are in bijection with transitive permutation actions of $F$ on based sets.  Indeed, given such a graph $Y$, as observed above the group $F$ acts by path lifting on the vertices.  The action of the generator $\alpha_i$ can be seen by restricting attention to the edges of $Y$ labelled by $\alpha_i$: by the covering condition, these edges form a union of topological circles; each circle corresponds to a cycle under the action of $\alpha_i$; taken together, these circles give the cycle type of $\alpha_i$.

Conversely, given an action of $F$ by permutations on a set $V$, we can build a covering graph $Y$.  Take $V$ to be the vertex set of $Y$.  There is an oriented edge labelled $\alpha_i$ from $u$ to $v$ if $\alpha_i.u=v$.  By construction, this graph $Y$ satisfies the covering condition.

\section{Marshall Hall's Theorem}\label{s: Hall's Theorem}

In this section, we recall Stallings's proof of Marshall Hall's Theorem \cite{stallings_topology_1983}.  The key observation is that one can complete any immersion to a covering map without increasing the number of vertices.

\begin{lemma}[Stallings \cite{stallings_topology_1983}]\label{l: Completion}
Let $Z$ be a finite graph, let $X$ be a rose and let $Z\to X$ be an immersion.  The immersion factors as $Z\hookrightarrow Y\to X$ where $Z$ is a subgraph of $Y$, every vertex of $Y$ is a vertex of $Z$, and $Y\to X$ is a covering map.
\end{lemma}
\begin{proof}
As above, we can think of $Z$ as an oriented, labelled graph that satisfies the immersion condition.  For each $i$, let $A^+(i)$ be the set of vertices of $Z$ that adjoin an outgoing edge labelled $\alpha_i$, and let $A^-(i)$ be the set of vertices of $Z$ that adjoin an incoming edge labelled $\alpha_i$.  Counting the edges labelled $\alpha_i$, we see that
\[
\#A^+(i)=\#A^-(i)
\]
whence
\[
\#(A^+(i))^c=\#(A^-(i))^c~.
\]
Choose any bijection between $(A^+(i))^c$ and $(A^-(i))^c$.  We can use this additional bijection to add new oriented edges labelled $\alpha_i$ to $Z$.  Let $Y$ be the result of carrying this out for each $i$.  By construction, $Y$ satisfies the covering condition.
\end{proof}

The proof is illustrated in Figure \ref{fig: Hall}.

\begin{figure}
\centering
\includegraphics[width=2in]{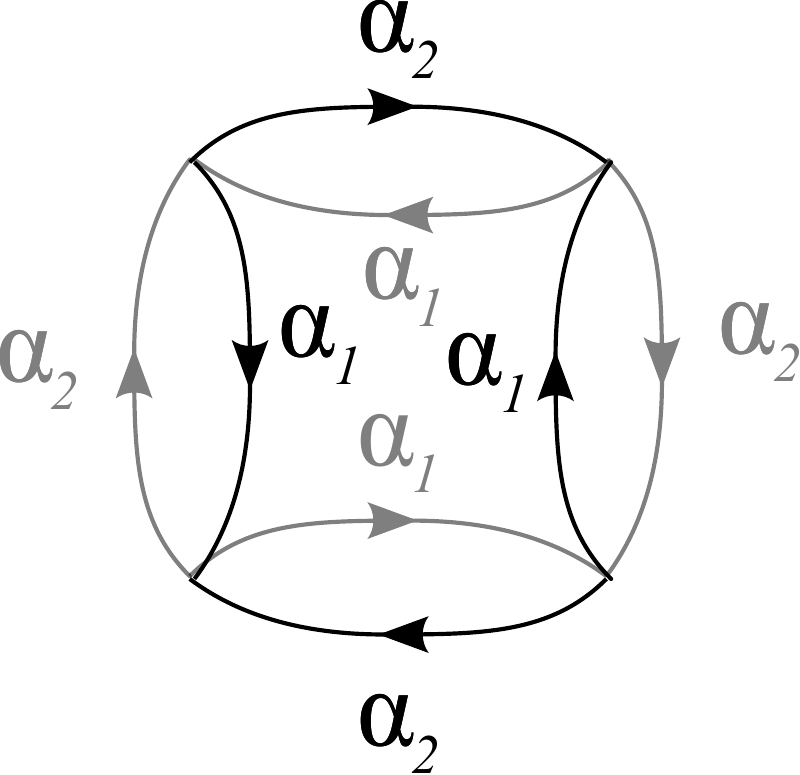}
\caption{The union of the black edges, $Z$, satisfies the immersion condition.  After adding the grey edges, the resulting graph $Y$ satisfies the covering condition.}
\label{fig: Hall}
\end{figure}

\begin{theorem}[Marshall Hall Jr \cite{hall_jr_subgroups_1949}]\label{t: Marshall Hall}
Let $F$ be a finitely generated free group, let $H\subseteq F$ be a finitely generated subgroup, and let $\gamma_1,\ldots,\gamma_n\in F\smallsetminus H$.  There is a homomorphism $f$ from $F$ to a finite group such that $f(\gamma_i)\notin f(H)$ for all $i$.
\end{theorem}
\begin{proof}
We identify $F$ with $\pi_1(X,x_0)$ for $X$ a suitable rose.  Let $(X',x'_0)\to (X,x_0)$ be the covering map corresponding to $H$.  Each $\gamma_i$ lifts at $x'_0$ to a path $\gamma'_i:[0,1]\to X'$; by assumption, $\gamma'_i(1)\neq x'_0$.  We now consider the topology of the graph $X'$.  Let $T\subseteq X'$ be a maximal tree.  Because $H$ is finitely generated, it follows from the Seifert--van Kampen Theorem that $X'\smallsetminus T$ is a finite union of edges $\epsilon'_1\cup\ldots\cup\epsilon'_k$.

For each $j$, let $\mu'_j$ be an edge path from $x'_0$ to the initial vertex of $\epsilon'_j$ and let $\nu'_j$ be an edge path from $x'_0$ to the terminal vertex of $\epsilon'_j$.  Let
\[
Z=\bigcup_{i=1}^n\gamma'_i\cup\bigcup_{j=1}^k (\mu'_j\cup\epsilon'_j\cup\nu'_j)~.
\]
That is, $Z$ is a finite, connected subgraph of $X'$ that contains the paths $\gamma'_i$ and that also carries the fundamental group of $X'$.  The labelling of $X'$ restricted to $Z$ satisfies the immersion condition, so we may apply Lemma \ref{l: Completion} to construct a finite-sheeted covering space $Y$.

Consider the action of $F$ on the vertices of $Y$.  Every element of $H$ is represented by a loop in $Z$.  Therefore, $H$ is contained in the stabiliser of $x'_0$.  On the other hand, $\gamma'_i(1)\neq x'_0$ for each $i$, and so $\gamma_i$ does not stabilise $x'_0$.
\end{proof}

\section{Alternating quotients}\label{s: Alternating quotients}

In this section, we show how to modify Stallings's construction to force the action of $F$ on the vertices of $Y$ to be alternating or symmetric.

Recall some basic terminology from the theory of symmetric group actions.  Consider a transitive action by a group $\Gamma$ on a finite set of order $n$, in other words a homomorphism $\Gamma\to S_n$ with transitive image; the integer $n$ is called the \emph{degree} of the action.   The smallest degree of a non-trivial cyclic subgroup of the image of $\Gamma$ is called the \emph{minimal degree} of the action.  The action is \emph{primitive} if it does not preserve any proper partition of the finite set.

\begin{remark}
Because the action is transitive, the cardinality of any partition preserved by $\Gamma$ divides $n$.  Therefore, if $n$ is prime then the action is primitive.
\end{remark}

We will use a classical theorem of Jordan, which gives a criterion for an action to be symmetric or alternating (in other words, for the homomorphism $\Gamma\to S_n$ to have image of index at most two) \cite{jordan_theoremes_1871,dieudonne_oeuvres_1961}.  See Theorem 3.3D of \cite{dixon_permutation_1996} for a modern treatment.

\begin{theorem}[Jordan]\label{t: Jordan's theorem}
There is a function $J:\mathbb{N}\to\mathbb{N}$ with the following properties:
\begin{enumerate}
\item $J(n)\to\infty$ as $n\to\infty$;
\item if $\Gamma\to S_n$ is primitive with minimal degree at most $J(n)$ then the image of $\Gamma$ is the symmetric group $S_n$ or the alternating group $A_n$.
\end{enumerate}
\end{theorem}

We can now provide the alternating analogue of Lemma \ref{l: Completion}.

\begin{figure}
\centering
\includegraphics[width=3.75in]{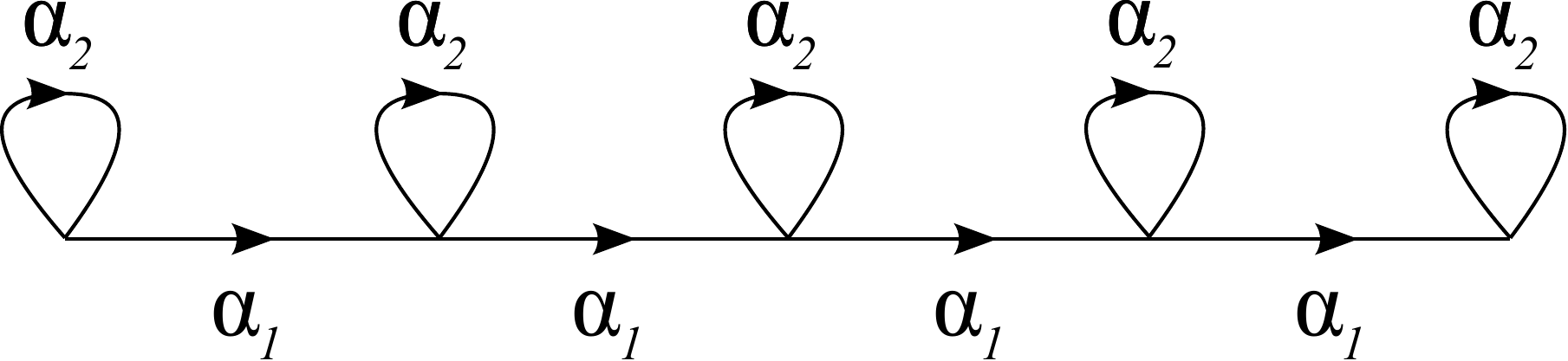}
\caption{The graph $W_4$.}
\label{fig: W_n}
\end{figure}

\begin{lemma}\label{l: Alternating completion}
Assume that $r\geq 2$.  Let $Z$ be a finite graph and let $X$ be a rose with $r$ petals.   Any immersion $Z\to X$ that is not a covering map factors as $Z\hookrightarrow Y\to X$ where $Y\to X$ is a finite-sheeted covering map and the action of $F\cong\pi_1(X)$ on the vertex set of $Y$ is alternating.
\end{lemma}
\begin{proof}
As before, we think of $Z$ as a labelled, oriented graph.  Because the immersion $Z\to X$ is not a covering map, the labelling and orientation on the graph $Z$ do not satisfy the covering condition, so for some $i$ there are (not necessarily distinct) vertices $a,b$ of  $Z$ such that $a$ does not adjoin an outgoing edge labelled $\alpha_i$ and $b$ does not adjoin an incoming edge labelled $\alpha_i$.  Without loss of generality, we may take $i=1$.  By adding edges to $Z$ as in the proof of Lemma \ref{l: Completion}, we may assume that $a$ and $b$ are the only vertices that adjoin fewer than $2r$ half-edges.  Let $d$ be the number of vertices of $Z$.

Before we proceed with the details of the proof, we will give an outline.  The labelled, oriented graph $Y$ is constructed from $Z$ by attaching a large labelled, oriented graph $W_n$ on which $\alpha_2$ acts trivially.  We need to ensure that $Y$ has the following properties:
\begin{enumerate}
\item the labelling and orientation on $Y$ satisfy the covering condition;
\item the number of vertices of $Y$ is prime (which ensures that the action of $F$ on the vertices of $Y$ is primitive);
\item the generator $\alpha_2$ fixes a large number of the vertices of $Y$ (this ensures that the action of $F$ on the vertices of $Y$ has small minimal degree).
\end{enumerate}
At this point it will follow from Theorem \ref{t: Jordan's theorem} that the action of $F$ on the vertices of $Y$ is either alternating or symmetric.  In order to force the action to be alternating, we will also attach a carefully chosen labelled, oriented graph $V_s$.  A schematic diagram of the construction of $Y$ is given in Figure \ref{fig: Schematic}.

We will now give the details of the construction of $W_n$.  Take a graph homeomorphic to an interval with $n+1$ vertices, denoted $w_0,\ldots,w_n$.  Label each edge $\alpha_1$ and orient the edges consistently, so that the immersion condition is satisfied.  Attach an oriented edge labelled $\alpha_i$ from $w_j$ to itself for each $i\neq 1$ and for each $j$.  The result is a labelled, oriented graph $W_n$, with $n$ vertices, that satisfies the immersion condition. Only two vertices of $W_n$ adjoin fewer than $2r$ half edges: one does not adjoin an incoming edge labelled $\alpha_1$ and the other does not adjoin an outgoing edge labelled $\alpha_1$.  An example is shown in Figure \ref{fig: W_n}.

Next, consider any sequence $(s_i)\in\{\pm 1\}$ (for $1\leq i\leq r)$.  We will define a labelled, oriented graph $V_s$, with four vertices, denoted $v_1,v_2,v_3,v_4$.  By changing the sequence $(s_i)$, we will be able to change the signs of the permutations defined by the generators of $F$. The construction is as follows.  There is an oriented edge labelled $\alpha_1$ from $v_1$ to $v_2$; for each $i>2$, there are oriented loops of length one attached to $v_1$ and $v_2$.  For each $i\neq 2$, if $s_i=+1$ then there are oriented loops of length one labelled $\alpha_i$ attached to $v_3$ and $v_4$; if $s_i=-1$ then $v_3$ and $v_4$ are the two vertices of an oriented loop of length two labelled $\alpha_i$.  If $s_2=+1$ then $v_1$ and $v_3$ are the two vertices of an oriented loop of length two labelled $\alpha_2$, and likewise $v_2$ and $v_4$ are the two vertices of such a loop; if $s_2=-1$ then $v_1,v_2,v_3,v_4$ are the vertices of an oriented loop of length four labelled $\alpha_2$.  Two examples are shown in Figure \ref{fig: V_s}.

\begin{figure}
\centering
\includegraphics[width=3.75in]{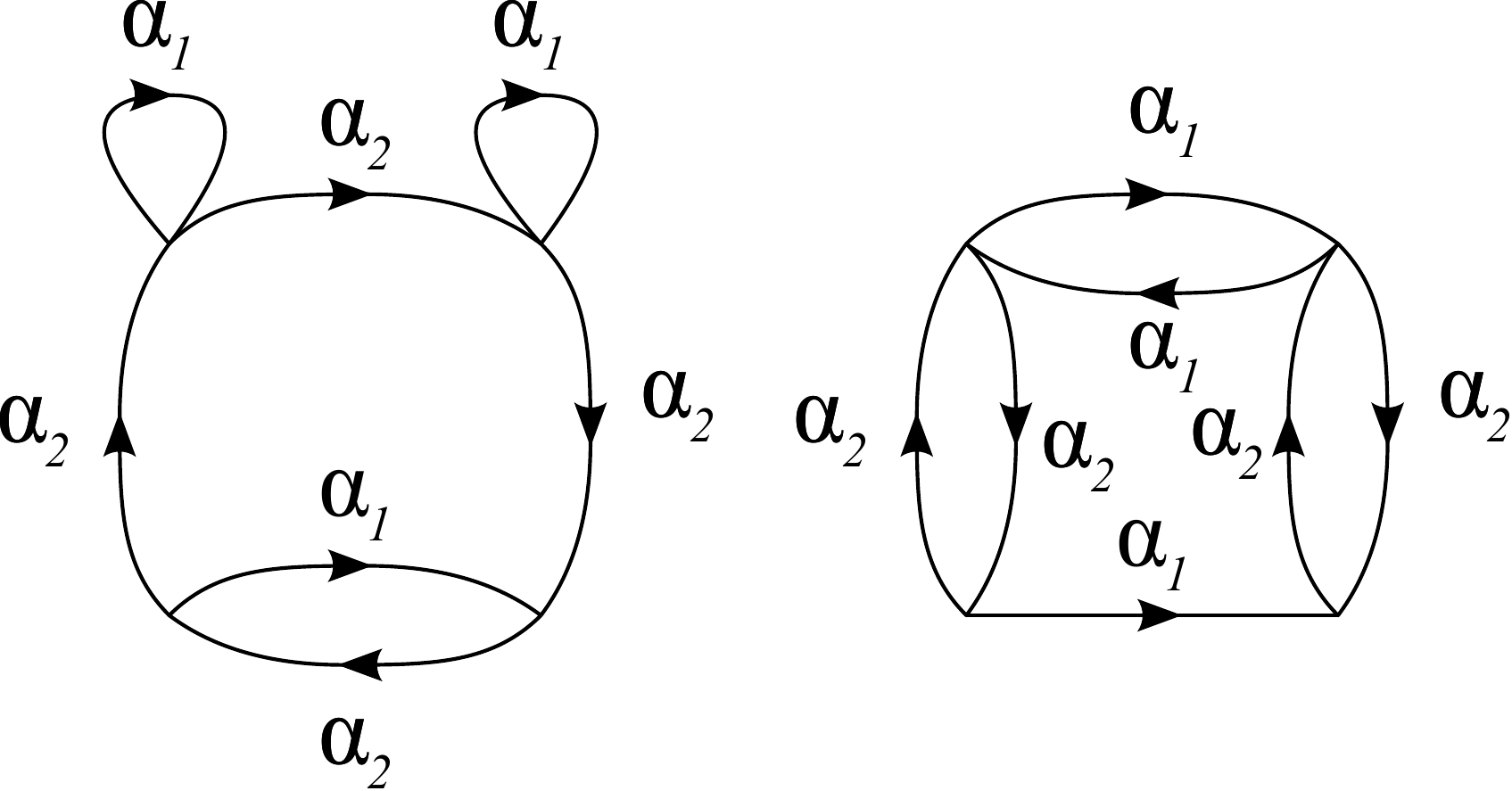}
\caption{The graph $V_{+1,-1}$ is shown on the left, and $V_{-1,+1}$ on the right.}
\label{fig: V_s}
\end{figure}

We are nearly ready to construct $Y$ from the pieces $Z$, $W_n$ and $V_s$.  However, we first construct a similar labelled, oriented graph $Y'$, in which we will not try to control the parity of the action of $F$.  Let $p$ be the smallest prime such that $J(p)\geq d+4$ and let $n=p-d-4$. Let $t_i=+1$ for $1\leq i\leq r$.  Construct $Y'$ from $V_t\sqcup Y\sqcup W_n$ by attaching four oriented edges labelled $\alpha_1$ in any way that makes the result connected and that satisfies the covering condition.  So $F$ acts on the vertex set of $Y'$, and we can read off the cycle types by looking at the edges. Because $Y'$ is connected, the action is transitive, and because the number of vertices $p$ is prime, it follows that the action is primitive.  By construction, $\alpha'_2.w_j=w_j$ for any $j$, and so we see that the minimal degree of the action is at most $d+4$.  It follows from Jordan's theorem that the action is alternating or symmetric.

\begin{figure}
\centering
\includegraphics[width=3.75in]{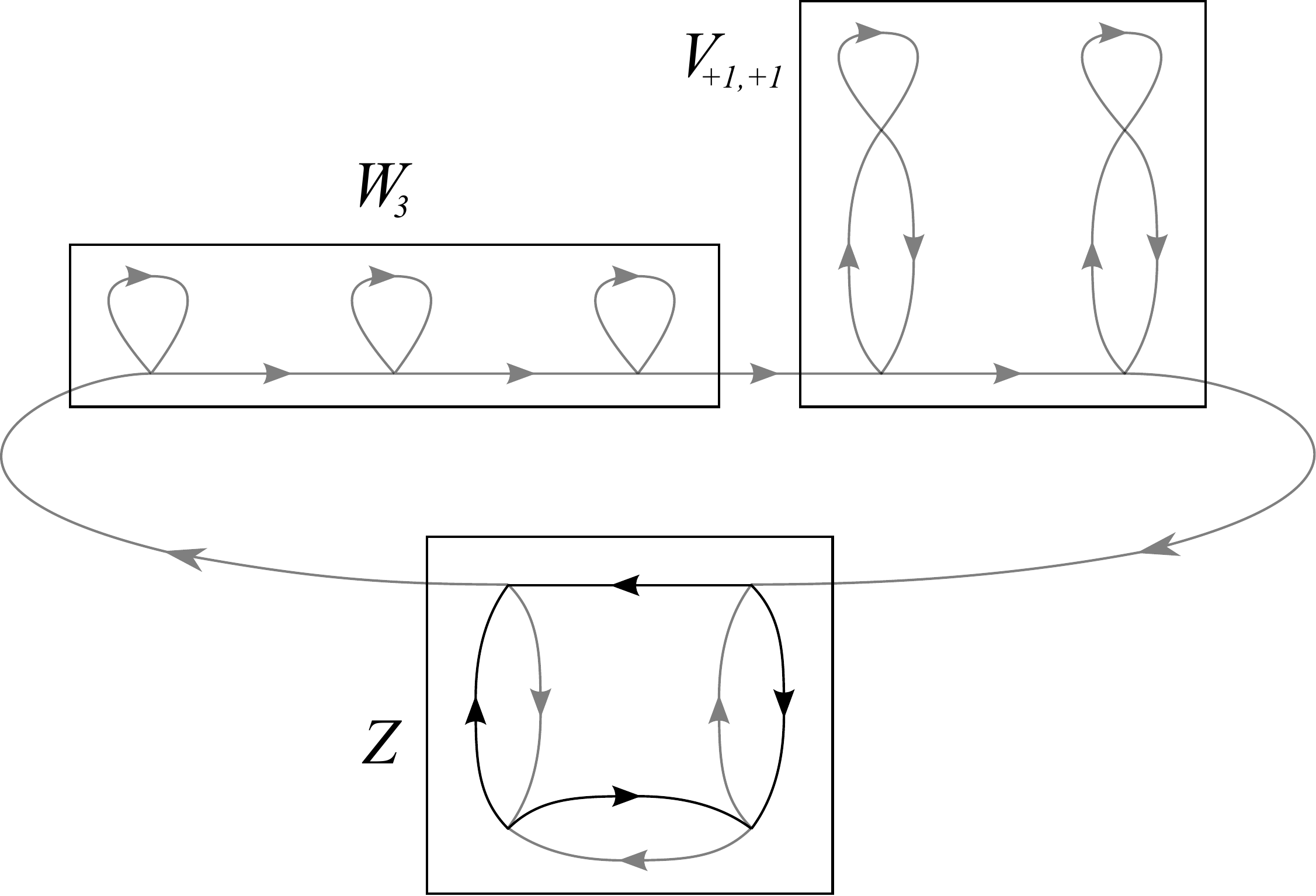}
\caption{The construction of $Y$ from $Z$ is illustrated.  The original edges of $Z$ are marked in black, and the added edges in grey.  The labels have been suppressed for clarity.}
\label{fig: Schematic}
\end{figure}

Finally, we will modify $Y'$ to force the action to be alternating.  For each $i$, consider the sign $s_i\in\{\pm 1\}$ of the action of $\alpha_i$ on the vertex set of $Y'$. Now construct $Y$ from $Y'$ by replacing $V_t$ by $V_s$.  As before, we see that the action of $F$ on the vertex set of $Y$ is symmetric or alternating, but the modification ensures that each generator $\alpha_i$ acts as an even permutation.  Therefore, the action of $F$ on the vertex set of $Y$ is alternating. 
\end{proof}

The construction of $Y$ is illustrated in Figure \ref{fig: Schematic}.  Note that, in order to ensure that the resulting action is symmetric rather than alternating, one can simply change the sequence $s$ in one place.  We therefore also have a symmetric version of the previous lemma.

\begin{lemma}\label{l: Symmetric completion}
Assume that $r\geq 2$.  Let $Z$ be a finite graph.   Any immersion $Z\to X$ that is not a covering map factors as $Z\hookrightarrow Y\to X$ where $Y\to X$ is a finite-sheeted covering map and the action of $F\cong\pi_1(X)$ on the vertex set of $Y$ is symmetric.
\end{lemma}

Theorem \ref{t: Alternating quotients} follows from Lemma \ref{l: Alternating completion} in exactly the same way that Theorem \ref{t: Marshall Hall} follows from Lemma \ref{l: Completion}.  Likewise, Theorem \ref{t: Symmetric quotients} follows from Lemma \ref{l: Symmetric completion}.

\bibliographystyle{plain}

\begin{thebibliography}{10}

\bibitem{dieudonne_oeuvres_1961}
Jean Dieudonn{\'e}.
\newblock {\em Oeuvres de {C}amille {J}ordan. {T}omes {I}, {II}}.
\newblock Publi\'ees sous la direction de M. Gaston Julia. Gauthier-Villars \&
  Cie, Paris, 1961.

\bibitem{dixon_permutation_1996}
John~D. Dixon and Brian Mortimer.
\newblock {\em Permutation groups}, volume 163 of {\em Graduate Texts in
  Mathematics}.
\newblock Springer-Verlag, New York, 1996.

\bibitem{haglund_finite_2008}
Fr\'ed\'eric Haglund.
\newblock Finite index subgroups of graph products.
\newblock {\em Geometriae Dedicata}, 135:167--209, 2008.

\bibitem{jordan_theoremes_1871}
Camille Jordan.
\newblock Th\'eor\`emes sur les groupes primitifs.
\newblock {\em J. Math. Pures. Appl.}, 16:383--408, 1871.

\bibitem{hall_jr_subgroups_1949}
Marshall~Hall Jr.
\newblock Subgroups of finite index in free groups.
\newblock {\em Canadian Journal of Mathematics}, 1:187--190, 1949.

\bibitem{katz_residual_1968}
Robert~A. Katz and Wilhelm Magnus.
\newblock Residual properties of free groups.
\newblock {\em Communications on Pure and Applied Mathematics}, 22:1--13, 1968.

\bibitem{long_simple_1998}
D.~D. Long and A.~W. Reid.
\newblock Simple quotients of hyperbolic 3-manifold groups.
\newblock {\em Proceedings of the American Mathematical Society},
  126(3):877--880, 1998.

\bibitem{pride_residual_1972}
Stephen~J. Pride.
\newblock Residual properties of free groups.
\newblock {\em Pacific Journal of Mathematics}, 43:725--733, 1972.

\bibitem{scott_subgroups_1978}
Peter Scott.
\newblock Subgroups of surface groups are almost geometric.
\newblock {\em Journal of the London Mathematical Society. Second Series},
  17(3):555--565, 1978.

\bibitem{stallings_topology_1983}
John~R. Stallings.
\newblock Topology of finite graphs.
\newblock {\em Inventiones Mathematicae}, 71(3):551--565, 1983.

\bibitem{weigel_residual_1993}
Thomas~S. Weigel.
\newblock Residual properties of free groups.
\newblock {\em Journal of Algebra}, 160(1):16--41, 1993.

\bibitem{wiegold_free_1977}
James Wiegold.
\newblock Free groups are residually alternating of even degree.
\newblock {\em Archiv der Mathematik}, 28(4):337--339, 1977.

\bibitem{wilton_halls_2008}
Henry Wilton.
\newblock Hall's {{T}heorem} for limit groups.
\newblock {\em Geometric and Functional Analysis}, 18(1):271--303, 2008.

\end{thebibliography}

\end{document}